\pdfoutput=1
\documentclass[11pt,a4paper,USenglish]{article}
\usepackage{babel}
\usepackage[T1]{fontenc}
\usepackage[utf8]{inputenc}
\usepackage{fullpage}
\usepackage[babel]{microtype}
\usepackage{csquotes}
\usepackage{booktabs}
\usepackage{authblk}
\usepackage{caption}
\usepackage[giveninits,sortcites,maxbibnames=10]{biblatex}
\usepackage{hyperref}
\usepackage{tikz}
\usepackage{titlesec}
\usepackage{enumitem}
\usepackage{amsmath}
\usepackage{mathtools}
\usepackage{amsthm}
\usepackage{thmtools}
\usepackage{dsfont}
\usepackage{amssymb}
\usepackage[noabbrev,capitalise,sort]{cleveref}

\hypersetup{pdfencoding=auto,colorlinks=true,linkcolor=blue,citecolor=teal,bookmarksnumbered=true,hypertexnames=false,pdfauthor={Elisabetta Carlini, Valentina Coscetti and Marco Pozza},pdftitle={Error Estimate for a Semi-Lagrangian Scheme for Hamilton-Jacobi Equations on Networks},pdfkeywords={error estimate;Hamilton-Jacobi equations;semi-Lagrangian scheme;embedded networks},pdfsubject={MSC(2020): 65M15, 49L25, 65M12, 35R02}}

\titlespacing{\paragraph}{0pt}{1em}{.7em}

\declaretheorem[style=plain,parent=section,title=Theorem,refname={Theorem,Theorems}]{theo}
\declaretheorem[style=plain,sibling=theo,title=Proposition,refname={Proposition,Propositions}]{prop}
\declaretheorem[style=plain,sibling=theo,title=Corollary,refname={Corollary,Corollaries}]{cor}
\declaretheorem[style=plain,sibling=theo,title=Lemma,refname={Lemma,Lemmas}]{lem}
\declaretheorem[style=definition,sibling=theo,title=Definition,refname={Definition,Definitions}]{defin}
\declaretheorem[style=remark,sibling=theo,title=Remark,refname={Remark,Remarks}]{rem}

\newlist{mylist}{itemize}{1}
\setlist[mylist]{label=\textbullet}
\newlist{Hassum}{enumerate}{1}
\setlist[Hassum]{label=\textbf{(H\arabic*)},ref=\textnormal{\textbf{(H\arabic*)}},font=\normalfont}
\newlist{Lassum}{enumerate}{1}
\setlist[Lassum]{label=\textbf{(L\arabic*)},ref=\textnormal{\textbf{(L\arabic*)}},font=\normalfont}
\newlist{mynum}{enumerate}{1}
\setlist[mynum]{label=\textbf{(\roman*)},font=\normalfont}

\crefname{Hassumi}{condition}{conditions}
\crefname{Lassumi}{condition}{conditions}
\crefname{equation}{}{}
\crefname{mynumi}{item}{items}

\usetikzlibrary{decorations.markings}

\newcommand{\Ecal}{\mathcal E}
\newcommand{\wtF}{\widetilde F}
\newcommand{\wtH}{\widetilde H}
\newcommand{\Hcal}{\mathcal H}
\newcommand{\wtHcal}{\widetilde\Hcal}
\newcommand{\Nds}{\mathds N}
\newcommand{\Ocal}{\mathcal O}
\newcommand{\ovOcal}{\overline\Ocal}
\newcommand{\Rds}{\mathds R}
\newcommand{\Scal}{\mathcal S}
\newcommand{\ovS}{\overline S}
\newcommand{\wtS}{\widetilde S}
\newcommand{\Tcal}{\mathcal T}
\newcommand{\ovT}{\overline T}
\newcommand{\Vbf}{\mathbf V}
\newcommand{\wtgamma}{\widetilde\gamma}
\newcommand{\whmu}{\widehat\mu}
\newcommand{\ovvarphi}{\overline\varphi}
\newcommand{\myemail}[2][]{\textsuperscript{#1}\href{mailto:#2}{\texttt{#2}}}

\title{Error Estimate for a Semi-Lagrangian Scheme for Hamilton--Jacobi Equations on Networks}
\author[1]{Elisabetta Carlini}
\author[1]{Valentina Coscetti}
\author[2]{Marco Pozza}
\affil[1]{Sapienza, University of Rome, Italy. \textit{Email addresses:} \myemail{carlini@mat.uniroma1.it}, \myemail{valentina.coscetti@uniroma1.it}}
\affil[2]{Link Campus University, Rome, Italy. \textit{Email address:} \myemail{m.pozza@unilink.it}}
\date{}

\begin{document}

    \maketitle

    \begin{abstract}
        We examine the numerical approximation of time-dependent Hamilton–Jacobi equations on networks, providing a convergence error estimate for the semi-Lagrangian scheme introduced in~\cite{CarliniSiconolfi23}, where convergence was proven without an error estimate. We derive a convergence error estimate of order one-half. This is achieved showing the equivalence between two definitions of solutions to this problem proposed in~\cite{ImbertMonneau17} and~\cite{Siconolfi22}, a result of independent interest, and applying a general convergence result from~\cite{CarliniFestaForcadel20}.
    \end{abstract}



    \section{Introduction}

    In this article, we focus on the numerical approximation of evolutive Hamilton--Jacobi (HJ) equations posed on networks. In recent years, the study of control problems on networks has gained significant attention due to its broad applicability in various domains, including data transmission, traffic management, flame propagation, and urban infrastructure planning, see~\cite{ImbertMonneauZidani12,CamilliFestaTozza17,CamilliCarliniMarchi18,AchdouMannucciMarchiTchou24}.

    The HJ equation has been extensively studied in classical settings, such as Euclidean unbounded or bounded domains. However, when posed on networks, the problem becomes significantly more challenging. Networks introduce unique complications, such as handling junction conditions at vertices, which represent the points where multiple edges meet. These challenges have led to the development of new mathematical techniques and suitable numerical methods.

    One key challenge in addressing HJ equations on networks is defining a well-posed solution, especially when each arc has unrelated Hamiltonians. This issue becomes particularly complex in evolutive problems, where discontinuities arise at the vertices due to the one-dimensional nature of the network. Properly handling these vertex discontinuities is critical for ensuring the well-posedness of the problem.

    Two main approaches have been proposed, both relying on the concept of flux limiters. The first approach, introduced by Imbert and Monneau in~\cite{ImbertMonneau17}, addresses the difficulties at the vertices by employing special test functions. These functions are defined across the entire network and act simultaneously on all arcs adjacent to a given vertex. In contrast, the second approach, proposed by Siconolfi in~\cite{Siconolfi22}, does not use special test functions. Instead, standard test functions are defined individually on each arc, acting separately and independently on them.

    The numerical approximation of the evolutive Hamilton--Jacobi equation in this context remains an active area of research. The inherent complexity of networks, especially the treatment of junctions, has resulted in the need for specialized approximation methods that can accurately and efficiently capture the dynamics.

    Several approaches have been proposed to tackle these challenges. In~\cite{CostesequeLebacqueMonneau14}, a finite difference scheme was introduced, and convergence was proven under a Courant--Friedrichs--Lewy (CFL) condition. An error estimate for this scheme was presented in~\cite{GuerandKoumaiha19}, achieving a convergence order of up to $\frac{1}{2}$, again under the CFL condition. In~\cite{CarliniFestaForcadel20}, a semi-Lagrangian scheme was developed to approximate evolutive problems with Hamiltonians that also depend on the state variable. This approach proved particularly effective, as a convergence result was obtained without restrictions on the time step. Moreover, under the same time restrictions imposed in~\cite{CostesequeLebacqueMonneau14}, the authors proved a rate of convergence of order $\frac{1}{2}$.

    More recently, in~\cite{CarliniSiconolfi23} a new semi-Lagrangian scheme was proposed, which exploits the viscosity solution framework introduced in~\cite{Siconolfi22}. In this framework, convergence is proven under the assumption that the ratio between the space step and the time step tends to zero, effectively an ``inverse CFL condition''. This scheme is particularly appealing due to its ability to allow large time steps while remaining explicit. Additionally, its structure---where computations are performed separately on each arc---makes it highly suitable for parallelization. However, a convergence error estimate for this scheme has not yet been proved.

    In this paper, we aim to provide a convergence error estimate for the scheme proposed in~\cite{CarliniSiconolfi23}. To achieve this, we first prove the equivalence between the two definitions of viscosity solutions introduced in~\cite{ImbertMonneau17} and~\cite{Siconolfi22}. We believe that this equivalence is of independent interest, as it bridges two important frameworks. Furthermore, it allows us to prove a new convergence result and derive a convergence error estimate for the scheme in~\cite{CarliniSiconolfi23}.

    The error estimate is obtained under the same time step restriction as those required in~\cite{CostesequeLebacqueMonneau14} for convergence, and in~\cite{GuerandKoumaiha19} and~\cite{CarliniFestaForcadel20} for error estimates. Taking advantage of the equivalence between the two viscosity solution definitions, we apply the general convergence result of~\cite{CarliniFestaForcadel20}, which provides an error estimate for any monotone, stable and consistent scheme, with the viscosity solution defined in~\cite{ImbertMonneau17}.

    The paper is organized as follows. In \cref{sec:intro}, we introduce the problem, present the main hypotheses, and provide the basic definitions. \Cref{Sec:equivalence} contains one of the key results of the paper: the proof of the equivalence between the viscosity solution definitions given in~\cite{ImbertMonneau17} and~\cite{Siconolfi22}. In \cref{sec:Scheme}, we describe the numerical scheme proposed in~\cite{CarliniSiconolfi23} and present our second main result: an error estimate for the scheme under a time step restriction. Finally, in \cref{sec:numerics}, we verify the numerical convergence in two test cases with different choices of time steps.

    \paragraph{Acknowledgments.}
    The first two authors were partially supported by Italian Ministry of Instruction, University and Research (MIUR) (PRIN Project2022238YY5, ``Optimal control problems: analysis, approximation'') and INdAM-research group GNCS (CUP$\_$ E53C23001670001, ``Metodi numerici innovativi per equazioni di Hamilton--Jacobi''). The first author was partially supported by KAUST through the subaward agreement ORA-2021-CRG10-4674.6. The third author is a member of the INdAM research group GNAMPA.

    \section{Basic Definitions and Assumptions}\label{sec:intro}

    We describe here our setting, namely the network and the time-dependent problem defined on it.

    \subsection{Networks}

    We fix a dimension $N$ and $\Rds^N$ as ambient space. An \emph{embedded network}, or \emph{continuous graph}, is a subset $\Gamma\subset\Rds^N$ of the form
    \begin{equation*}
        \Gamma=\bigcup_{\gamma\in\Ecal}\gamma([0,|\gamma|])\subset\Rds^N,
    \end{equation*}
    where $\Ecal$ is a finite collection of regular (i.e., $C^1$ with non-vanishing derivative) simple oriented curves, called \emph{arcs} of the network, with Euclidean length $|\gamma|$ and parameterized by arc length in $[0,|\gamma|]$ (i.e., $|\dot\gamma|\equiv1$ for any $\gamma\in\Ecal$). Note that we are also assuming existence of one-sided derivatives at the endpoints $0$ and $|\gamma|$. We stress out that a regular change of parameters does not affect our results.

    Observe that on the support of any arc $\gamma$, we also consider the inverse parametrization defined as
    \begin{equation*}
        \widetilde\gamma(s)\coloneqq\gamma(|\gamma|-s),\qquad\text{for }s\in[0,|\gamma|].
    \end{equation*}
    We call $\widetilde\gamma$ the \emph{inverse arc} of $\gamma$. We assume
    \begin{equation}\label{eq:nosovrap}
        \gamma((0,|\gamma|))\cap\gamma'\left(\left(0,\left|\gamma'\right|\right)\right)=\emptyset,\qquad\text{whenever }\gamma'\ne\gamma,\widetilde\gamma.
    \end{equation}

    We call \emph{vertices} the initial and terminal points of the arcs, and denote by $\Vbf$ the sets of all such vertices. Note that~\eqref{eq:nosovrap} implies
    \begin{equation*}
        \gamma((0,|\gamma|))\cap\Vbf=\emptyset,\qquad\text{for any }\gamma\in\Ecal.
    \end{equation*}

    We assume that it is given an \emph{orientation} of $\Gamma$, i.e., a subset $\Ecal^+\subset\Ecal$ containing exactly one arc in each pair $\{\gamma,\wtgamma\}$. We further set, for each $x\in\Gamma$,
    \begin{equation*}
        \Ecal_x^+\coloneqq\left\{\gamma\in\Ecal^+:\text{$\gamma(s)=x$ for some $s\in[0,|\gamma|]$}\right\}.
    \end{equation*}
    Observe that $\Ecal_x^+$ is a singleton when $x\in\Gamma\setminus\Vbf$, while it contains in general more than one element if $x$ is instead a vertex.

    We assume that the network is connected, namely given two vertices there is a finite concatenation of arcs linking them. A \emph{loop} is an arc with initial and final point coinciding. The unique restriction we require on the geometry of $\Gamma$ is
    \begin{mylist}
        \item $\Ecal$ does not contain loops.
    \end{mylist}
    This condition is due to the fact that in the known literature about time-dependent Hamilton--Jacobi equations on networks no loops are admitted, see, e.g., \cite{ImbertMonneau17,Siconolfi22,LionsSouganidis17}.

    We say that a continuous function $\Phi:\Gamma\to\Rds$ belongs to $C^i(\Gamma)$, with $i\in\{1,2\}$, if $\Phi\circ\gamma\in C^i([0,|\gamma|])$ for any $\gamma\in\Ecal$. We will also use the notation
    \begin{equation*}
        \partial\Phi(x)\coloneqq\{D(\Phi\circ\gamma)(0)\}_{\{\gamma:\gamma(0)=x\}},\qquad\text{for }x\in\Vbf.
    \end{equation*}
    This definition can be trivially extended to functions on $\Gamma\times\Rds^+$.

    \subsection{Time-Dependent Hamilton--Jacobi Equations}

    A Hamiltonian on $\Gamma$ is a collection of Hamiltonians $\wtHcal\coloneqq\left\{\wtH_\gamma\right\}_{\gamma\in\Ecal}$, where
    \begin{alignat*}{2}
        \wtH_\gamma:[0,|\gamma|]\times\Rds&&\:\longrightarrow\:&\Rds\\
        (s,\mu)&&\:\longmapsto\:&\wtH_\gamma(s,\mu),
    \end{alignat*}
    satisfying
    \begin{equation*}
        \wtH_{\wtgamma}(s,\mu)=\wtH_{\gamma}(|\gamma|-s,-\mu),\qquad\text{for any arc }\gamma.
    \end{equation*}
    We emphasize that, apart the above compatibility condition, the Hamiltonians $\wtH_\gamma$ are \emph{unrelated}.

    We require each $\wtH_\gamma$ to be:
    \begin{Hassum}
        \item continuous in both arguments;
        \item\label{condconv} convex in $\mu$;
        \item\label{condcoerc} coercive in $\mu$, uniformly in $s$.
    \end{Hassum}

    We consider the following time-dependent problem on $\Gamma$:
    \begin{equation}\label{eq:globteik}\tag{$\Hcal$J}
        \partial_t u(x,t)+\wtHcal(x,D_x u)=0,\qquad\text{on }\Gamma\times(0,\infty).
    \end{equation}
    This notation synthetically indicates the family (for $\gamma$ varying in $\Ecal$) of Hamilton--Jacobi equations
    \begin{equation}\label{eq:teikg}\tag{HJ\textsubscript{$\gamma$}}
        \partial_t U(s,t)+\wtH_\gamma(s,\partial_s U(s,t))=0,\qquad\text{on }(0,|\gamma|)\times(0,\infty).
    \end{equation}

    Here (sub/super)solutions to the local problem~\eqref{eq:teikg} are intended in the viscosity sense, see for instance~\cite{BardiCapuzzo-Dolcetta97} for a comprehensive treatment of viscosity solutions theory. We just recall that, given an open set $\Ocal$ and a continuous function $u:\ovOcal\to\Rds$, a \emph{supertangent} (resp.\ \emph{subtangent}) to $u$ at $x\in\Ocal$ is a viscosity test function from above (resp.\ below). We say that a subtangent $\varphi$ to $u$ at $x\in\partial\Ocal$ is \emph{constrained to $\ovOcal$} if $x$ is a minimizer of $u-\varphi$ in a neighborhood of $x$ intersected with $\ovOcal$.

    We call \emph{flux limiter} any function $x\mapsto c_x$ from $\Vbf$ to $\Rds$ satisfying
    \begin{equation}\label{eq:fluxlimdef}
        c_x\le-\max_{\gamma\in\Ecal^+_x,\,s\in[0,|\gamma|]}\min_{\mu\in\Rds}\wtH_\gamma(s,\mu),\qquad\text{for any }x\in\Vbf.
    \end{equation}

    The definition of (sub/super)solutions to~\eqref{eq:globteik} given in~\cite{Siconolfi22} is as follows:

    \begin{defin}\label{deftsubsol}
        We say that a continuous function $w:\Gamma\times\Rds^+\to\Rds$ is a \emph{viscosity subsolution} to~\eqref{eq:globteik} if
        \begin{mynum}
            \item $(s,t)\mapsto w(\gamma(s),t)$ is a viscosity subsolution to~\eqref{eq:teikg} in $(0,|\gamma|)\times(0,\infty)$ for any $\gamma\in\Ecal^+$;
            \item for any $t_0\in(0,\infty)$ and vertex $x$, if $\psi$ is a $C^1$ supertangent to $w(x,\cdot)$ at $t_0$ then $\partial_t\psi(t_0)\le c_x$.
        \end{mynum}
        A continuous function $v:\Gamma\times\Rds^+\to\Rds$ is called a \emph{viscosity supersolution} to~\eqref{eq:globteik} if
        \begin{mynum}[resume]
            \item $(s,t)\mapsto v(\gamma(s),t)$ is a viscosity supersolution to~\eqref{eq:teikg} in $(0,|\gamma|)\times(0,\infty)$ for any $\gamma\in\Ecal^+$;
            \item\label{stateconst} for every vertex $x$ and $t_0\in(0,\infty)$, if $\psi$ is a $C^1$ subtangent to $v(x,\cdot)$ at $t_0$ such that $\partial_t\psi(t_0)<c_x$, then there is a $\gamma$ with $\gamma(0)=x$ such that
            \begin{equation*}
                \partial_t\varphi(0,t_0)+H_\gamma(0,\partial_s\varphi(0,t_0))\ge0
            \end{equation*}
            for any $\varphi$ that is a constrained $C^1$ subtangent to $(s,t)\mapsto v(\gamma(s),t)$ at $(0,t_0)$. We stress out that this condition does not require the existence of constrained subtangents.
        \end{mynum}
        Furthermore, we say that $u:\Gamma\times\Rds^+\to\Rds$ is a \emph{viscosity solution} to~\eqref{eq:globteik} if it is both a viscosity subsolution and supersolution.
    \end{defin}

    We also have a result concerning the existence of solutions.

    \begin{theo}\label{exunsolt}
        \emph{\cite[Proposition 7.8]{Siconolfi22}} Given $u_0\in C(\Gamma)$ and a flux limiter $c_x$, \eqref{eq:globteik} admits a unique solution with initial datum $u_0$ and flux limiter $c_x$. If $u_0$ is Lipschitz continuous, the solution is Lipschitz continuous as well.
    \end{theo}

    \section{An Equivalent Notion of Solution}\label{Sec:equivalence}

    The main result of this paper is an error estimate for the semi-Lagrangian scheme described in~\cite{CarliniSiconolfi23} which approximate solutions, in the sense of \cref{deftsubsol}, to~\eqref{eq:globteik}. To do so we will employ a result given in~\cite{CarliniFestaForcadel20} for numerical schemes approximating a different kind of solution proposed in~\cite{ImbertMonneau17}. Thereby our first step is to show that these different notions of solutions to~\eqref{eq:globteik} are equivalent. Although these definitions are similar, to our knowledge there is no previous result showing their equivalence.

    We introduce the functions
    \begin{equation}\label{eq:minHdef}
        \wtH^-_\gamma(s,\mu)\coloneqq\left\{
        \begin{aligned}
            &\wtH_\gamma(s,\mu),&&\text{for }\mu\le\whmu_\gamma,\,s\in[0,|\gamma|]\\
            &\wtH_\gamma(s,\whmu_\gamma),&&\text{for }\mu>\whmu_\gamma,\,s\in[0,|\gamma|],
        \end{aligned}
        \right.\qquad\text{for }\gamma\in\Ecal,
    \end{equation}
    where $\whmu_\gamma$ is a minimizer of $\mu\mapsto\wtH_\gamma(0,\mu)$. We point out that such minimizer exists by \cref{condconv,condcoerc}. Fixed a flux limiter $c_x$, we then set for $x\in\Vbf$ and $p=\{p_\gamma\}_{\{\gamma:\gamma(0)=x\}}$, where each $p_\gamma$ is parallel to $\dot\gamma(0)$,
    \begin{equation}\label{eq:Fvertdef}
        \wtF(x,p)\coloneqq\max\left\{-c_x,\max_{\{\gamma:\gamma(0)=x\}}\wtH^-_\gamma(0,p_\gamma\dot\gamma(0))\right\}.
    \end{equation}
    We further recall the definition of upper and lower semicontinuous envelopes $u^*$ and $u_*$ of a locally bounded function $u:\Gamma\times\Rds^+\to\Rds$:
    \begin{equation*}
        u^*(x,t)\coloneqq\limsup_{(y,r)\to(x,t)}u(y,r),\qquad u_*(x,t)\coloneqq\liminf_{(y,r)\to(x,t)}u(y,r).
    \end{equation*}

    \begin{defin}\label{IMdef}
        \cite[Definition 5.5]{ImbertMonneau17} We say that $w$ is a \emph{flux limited subsolution} to~\eqref{eq:globteik} if
        \begin{mynum}
            \item $(s,t)\mapsto w^*(\gamma(s),t)$ is a viscosity subsolution to~\eqref{eq:teikg} in $(0,|\gamma|)\times(0,\infty)$ for any $\gamma\in\Ecal$;
            \item fixed a flux limiter $c_x$, for all $(x,t_0)\in\Vbf\times(0,\infty)$ and $C^1\left(\Gamma\times\Rds^+\right)$ supertangent $\Psi$ to $w^*$ at $(x,t_0)$
            \begin{equation*}
                \partial_t\Psi(x,t_0)+\wtF(x,\partial_x\Psi(x,t_0))\le0.
            \end{equation*}
        \end{mynum}
        We say that $v$ is a \emph{flux limited supersolution} to~\eqref{eq:globteik} if
        \begin{mynum}[resume]
            \item\label{IMsupsolloc} $(s,t)\mapsto v_*(\gamma(s),t)$ is a viscosity supersolution to~\eqref{eq:teikg} in $(0,|\gamma|)\times(0,\infty)$ for any $\gamma\in\Ecal$;
            \item\label{IMstateconst} fixed a flux limiter $c_x$, for all $(x,t_0)\in\Vbf\times(0,\infty)$ and $C^1\left(\Gamma\times\Rds^+\right)$ subtangent $\Phi$ to $v_*$ at $(x,t_0)$
            \begin{equation*}
                \partial_t\Phi(x,t_0)+\wtF(x,\partial_x\Phi(x,t_0))\ge0.
            \end{equation*}
        \end{mynum}
        Furthermore, we say that $u$ is a \emph{flux limited solution} to~\eqref{eq:globteik} if it is both a flux limited subsolution and supersolution.
    \end{defin}

    \begin{theo}\label{flsolexun}
        \emph{\cite[Corollary 5.9]{ImbertMonneau17}} Let $c_x$ be a flux limiter and $u_0\in C(\Gamma)$. There is a unique locally bounded solution to~\eqref{eq:globteik}, in the sense of \cref{IMdef}, with initial datum $u_0$ and flux limiter $c_x$.
    \end{theo}

    As pointed out in~\cite{Siconolfi22}, \cref{IMdef,deftsubsol} are closely related. Indeed, according to \cite[Theorem 2.11]{ImbertMonneau17}, the notion of continuous subsolution is the same:

    \begin{prop}\label{equivsubsol}
        A continuous function is a subsolution to~\eqref{eq:globteik} in the sense of \cref{deftsubsol} if and only it is a subsolution in the sense of \cref{IMdef}.
    \end{prop}

    We can actually show that, in our case, these notions of solution are equivalent:

    \begin{theo}\label{defequiv}
        A locally bounded function is a solution to~\eqref{eq:globteik}, with flux limiter $c_x$ and continuous initial datum $u_0$, in the sense of \cref{deftsubsol} if and only if it is a solution in the sense of \cref{IMdef}.
    \end{theo}
    \begin{proof}
        We will show that a continuous solution $u$ in the sense of \cref{deftsubsol} is also a solution in the sense of \cref{IMdef}, then the existence of the former type of solution (\cref{exunsolt}) and the uniqueness of the latter when the initial datum is continuous (\cref{flsolexun}) will imply our claim.\\
        That the notions of continuous subsolution are equivalent is stated in \cref{equivsubsol} and a supersolution in the sense of \cref{deftsubsol} clearly satisfies \labelcref{IMsupsolloc} in \cref{IMdef}, thus we only need to prove that \labelcref{IMstateconst} in \cref{IMdef} holds for $u$. Let $\Phi\in C^1(\Gamma\times\Rds^+)$ be a subtangent to $u$ at a point $(x,t_0)\in\Vbf\times(0,\infty)$. If $\partial_t\Phi(x,t_0)\ge c_x$ then
        \begin{equation}\label{eq:defequiv1}
            0\le\partial_t\Phi(x,t_0)-c_x\le\partial_t\Phi(x,t_0)+\wtF(x,\partial_x\Phi(x,t_0)).
        \end{equation}
        If instead
        \begin{equation}\label{eq:defequiv2}
            \partial_t\Phi(x,t_0)<c_x,
        \end{equation}
        we have from \cref{deftsubsol}\ref{stateconst} that there is an arc $\gamma$ with $\gamma(0)=x$ such that
        \begin{equation}\label{eq:defequiv3}
            \partial_t\varphi(0,t_0)+\wtH_\gamma(0,D_s\varphi(0,t_0))\ge0,
        \end{equation}
        for any constrained $C^1$ subtangent $\varphi$ to $(s,t)\mapsto u(\gamma(s),t)$ at $(0,t_0)$. Clearly $(s,t)\mapsto\Phi(\gamma(s),t)$ is a constrained $C^1$ subtangent to $(s,t)\mapsto u(\gamma(s),t)$ at $(0,t_0)$, thus, if $D_s\Phi(\gamma(0),t_0)\le\whmu_\gamma$ (see~\eqref{eq:minHdef}), \eqref{eq:defequiv3} implies
        \begin{equation}\label{eq:defequiv4}
            0\le\partial_t\Phi(\gamma(0),t_0)+\wtH_\gamma^-(0,D_s\Phi(\gamma(0),t_0))\le\partial_t\Phi(x,t_0)+\wtF(x,\partial_x\Phi(x,t_0)).
        \end{equation}
        Since $(x,t_0)$ and $\Phi$ are arbitrary, and in view of~\eqref{eq:defequiv1} and~\eqref{eq:defequiv4}, to prove our claim it is enough to show that if~\eqref{eq:defequiv2} holds, then $D_s\Phi(\gamma(0),t_0)\le\whmu_\gamma$. We argue by contradiction, assuming that there is a $\delta>0$ such that
        \begin{equation}\label{eq:defequiv5}
            D_s\Phi(\gamma(s),t)>\whmu_\gamma,\qquad\text{for any }(s,t)\in[0,\delta]\times[t_0-\delta,t_0+\delta].
        \end{equation}
        Setting
        \begin{equation*}
            \ovvarphi(s,t)\coloneqq\Phi(\gamma(0),t)+\whmu_\gamma s,\qquad\text{for }(s,t)\in[0,\delta]\times[t_0-\delta,t_0+\delta],
        \end{equation*}
        \eqref{eq:defequiv5} yields
        \begin{align*}
            \ovvarphi(0,t_0)=\;&\Phi(\gamma(0),t_0)=u(\gamma(0),t_0),\\
            \ovvarphi(s,t)\le\;&\Phi(\gamma(s),t)\le u(\gamma(s),t),&&\text{in }[0,\delta]\times[t_0-\delta,t_0+\delta],
        \end{align*}
        i.e., $\ovvarphi$ is a constrained $C^1$ subtangent to $(s,t)\mapsto u(\gamma(s),t)$ at $(0,t_0)$. It follows from~\eqref{eq:fluxlimdef} and~\eqref{eq:defequiv2} that
        \begin{equation*}
            \partial_t\ovvarphi(0,t_0)+\wtH_\gamma(0,\partial_s\ovvarphi(0,t_0))=\partial_t\Phi(\gamma(0),t_0)+\wtH_\gamma(0,\whmu_\gamma)<c_x+\max_{s\in[0,|\gamma|]}\min_{\mu\in\Rds}\wtH_{\gamma}(s,\mu)\le0,
        \end{equation*}
        which is in contradiction with~\eqref{eq:defequiv3}.
    \end{proof}

    \section{Approximation of Time-Dependent Hamilton--Jacobi Equations}\label{sec:Scheme}

    Here we illustrate the semi-Lagrangian scheme introduced in~\cite{CarliniSiconolfi23} for the approximation of solutions to~\eqref{eq:globteik} and provide an error estimate.

    We require the following additional conditions on each $\wtH_\gamma$:
    \begin{Hassum}[resume]
        \item\label{condlip} $s\mapsto\wtH_\gamma(s,\mu)$ is Lipschitz continuous for any $\mu\in\Rds$;
        \item $\lim\limits_{|\mu|\to\infty}\inf\limits_{s\in[0,|\gamma|]}\dfrac{\wtH_\gamma(s,\mu)}{|\mu|}=\infty$.
    \end{Hassum}
    We point out that the $\wtH_\gamma$ are locally Lipschitz continuous by~\ref{condconv}, \labelcref{condlip} and \cite[Corollary to Proposition 2.2.6]{Clarke90}.

    \subsection{Modified Problem}

    Actually, the semi-Lagrangian scheme does not directly perform an approximation of~\eqref{eq:globteik}, but instead of a problem with modified Hamiltonians $H_\gamma$ obtained from the $\wtH_\gamma$ through the procedure described in \cite[Appendix A]{CarliniSiconolfi23}, for a suitable preliminary choice of a compact interval $I$ of the momentum variable $\mu$ with
    \begin{equation*}
        H_\gamma(s,\mu)=\wtH_\gamma(s,\mu),\qquad\text{for every }\gamma\in\Ecal,\,s\in[0,|\gamma|],\,\mu\in I.
    \end{equation*}
    The advantage is that a positive constant $\beta_0$, depending on $I$, can be found so that the modified Lagrangians
    \begin{equation*}
        L_\gamma(s,\lambda)\coloneqq\max_{\mu\in\Rds}(\lambda\mu-H_\gamma(s,\mu)),\qquad\text{for }\gamma\in\Ecal,\,s\in [0,|\gamma|],\,\lambda\in\Rds,
    \end{equation*}
    are infinite when $\lambda$ is outside $[-\beta_0,\beta_0]$, for all $s\in[0,|\gamma|]$, $\gamma\in\Ecal$, and Lipschitz continuous in $[-\beta_0,\beta_0]$. $H_\gamma$ is the convex conjugate of $L_\gamma$, thus we get the following \namecref{modHprop}.

    \begin{lem}\label{modHprop}
        Each $H_\gamma$ is uniformly Lipschitz continuous and satisfies \cref{condcoerc,condconv}.
    \end{lem}
    \begin{proof}
        That each $H_\gamma$ satisfies \cref{condcoerc,condconv} is shown in~\cite{CarliniSiconolfi23}. The Lipschitz continuity is a consequence of the following identity:
        \begin{equation*}
            H_\gamma(s,\mu)=\max_{|\lambda|\le\beta_0}(\lambda\mu-L_\gamma(s,\lambda)),\qquad\text{for any }\gamma\in\Ecal,\,(s,\mu)\in[0,|\gamma|]\times\Rds.
        \end{equation*}
    \end{proof}

    Using the $H_\gamma$ and fixed a flux limiter $c_x$, we further define the $H_\gamma^-$ and $F$ as in~\eqref{eq:minHdef} and~\eqref{eq:Fvertdef}, respectively.

    Setting $\Hcal\coloneqq\{H_\gamma\}_{\gamma\in\Ecal}$, we consider the problem
    \begin{equation}\tag{$\Hcal$Jmod}\label{eq:globteikmod}
        \partial_t u(x,t)+\Hcal(x,D_x u)=0,\qquad\text{on }\Gamma\times(0,\infty),
    \end{equation}
    in place of~\eqref{eq:globteik}. The rationale behind this change is due to the next \namecref{solmod}.

    \begin{theo}\label{solmod}
        Let $c_x$ be a flux limiter and $u_0$ be a Lipschitz continuous initial datum. For a suitable choice of the interval $I$, which depends upon $c_x$ and $u_0$, the solutions to~\eqref{eq:globteik} and~\eqref{eq:globteikmod} with flux limiter $c_x$ and initial datum $u_0$ coincide.
    \end{theo}
    \begin{proof}
        A similar result is proved in \cite[Theorem 2.7]{CarliniSiconolfi23} for problems defined on a finite time interval $(0,T)$. Since the time $T$ is arbitrary and the modified Hamiltonians do not depend on it, this proves our claim.
    \end{proof}

    \subsection{Discretization}

    We describe the space-time discretization on which the approximation scheme is based.

    We start by fixing $T>0$, a spatial and a time step denoted by $\Delta x$ and $\Delta t$, respectively, and setting $\Delta\coloneqq(\Delta x,\Delta t)$. We will say that $\Delta$ is an \emph{admissible pair} for $T$ if
    \begin{equation*}
        0<\Delta x<|\gamma|,\quad\text{for any }\gamma\in\Ecal^+,\qquad\text{and}\qquad0<\Delta t<T.
    \end{equation*}

    We further define, for any admissible $\Delta$,
    \begin{equation*}
        N^\Delta_\gamma\coloneqq\left\lceil\frac{|\gamma|}{\Delta x}\right\rceil,\quad\text{for }\gamma\in\Ecal^+,\qquad\text{and}\qquad N^\Delta_T\coloneqq\left\lceil\frac T{\Delta t}\right\rceil,
    \end{equation*}
    where $\lceil\cdot\rceil$ stands for the ceiling function.

    Next we set
    \begin{equation*}
        s^{\Delta,\gamma}_i\coloneqq i\frac{|\gamma|}{N^\Delta_\gamma},\quad\text{for }i\in\left\{0,\dotsc,N_\gamma^\Delta\right\},\,\gamma\in\Ecal^+,\qquad t^\Delta_m\coloneqq m\frac T{N^\Delta_T},\quad\text{for }m\in\left\{0,\dotsc,N_T^\Delta\right\}.
    \end{equation*}
    We have therefore associated to $\Delta$ and any arc $\gamma\in\Ecal^+$, the partition of the parameter interval $[0,|\gamma|]$ in $N^\Delta_\gamma$ subintervals with length less than or equal to $\Delta x$, denoted by $\Delta_\gamma x$, and similarly a partition of the time interval $[0,T]$ in $N_T^\Delta$ subintervals all of them with length less than or equal to $\Delta t$, denoted by $\overline{\Delta t}$.

    \begin{rem}
        It is apparent that each $\Delta_\gamma x$ and $\overline{\Delta t}$ have the same order of $\Delta x$ and $\Delta t$, respectively. Hereafter we will rarely mention $\Delta_\gamma x$, since most of the time we will just use the fact that $\Delta_\gamma x\le\Delta x$ without explicitly mentioning. Conversely, $\overline{\Delta t}$ plays a central role in the definition of the numerical scheme and in our analysis, while $\Delta t$ is actually never used. Therefore, for the sake of notation, we just assume that $\overline{\Delta t}=\Delta t$.
    \end{rem}

    We proceed introducing the space-time grids associated to any $\gamma\in\Ecal^+$:
    \begin{equation*}
        \Scal_{\Delta,\gamma}\coloneqq\left\{s_i^{\Delta,\gamma}:i\in\left\{0,\dotsc,N_\gamma^\Delta\right\}\right\},\qquad\Tcal_\Delta^T\coloneqq\left\{t^\Delta_m:m\in\left\{0,\dotsc,N^\Delta_T\right\}\right\}
    \end{equation*}
    and
    \begin{equation*}
        \Gamma^0_\Delta\coloneqq\bigcup_{\gamma\in\Ecal^+}\gamma(\Scal_{\Delta,\gamma}),\qquad\Gamma_\Delta^T\coloneqq\Gamma_\Delta^0\times\Tcal^T_\Delta.
    \end{equation*}

    To ease notation, henceforth we omit the index $\Delta$ from the above formulas.

    We denote by $B\left(\Gamma^0\right)$ and $B(\Scal_\gamma)$, for $\gamma\in\Ecal^+$, the spaces of functions from $\Gamma^0$ and $\Scal_\gamma$, respectively, to $\Rds$.

    Given an arc $\gamma\in\Ecal^+$ and $f\in B(\Scal_\gamma)$ we define the operator
    \begin{equation}\label{eq:S_gamma}
        S_\gamma[f](s)\coloneqq\min_{\frac{s-|\gamma|}{\Delta t}\le\lambda\le\frac s{\Delta t}}\{I_\gamma[f](s-\Delta t\lambda)+\Delta tL_\gamma(s,\lambda)\},\qquad\text{for }s\in\Scal_\gamma,
    \end{equation}
    where $I_\gamma$ is the \emph{interpolating polynomial} of degree 1 defined by
    \begin{equation*}
        I_\gamma[f](s)\coloneqq f(s_i)+\frac{s-s_i}{s_{i+1}-s_i}(f(s_{i+1})-f(s_i)),
    \end{equation*}
    with $s_i,s_{i+1}\in\Scal_\gamma$ and $s\in[s_i,s_{i+1}]$. We recall the following result for the linear interpolating polynomial in dimension 1 (see, e.g., \cite[Lemma 4.2]{CarliniSiconolfi23}):

    \begin{lem}\label{lipinterpol}
        If $f$ is Lipschitz continuous, then $I_\gamma[f]$ is Lipschitz continuous as well with the same Lipschitz constant.
    \end{lem}

    We extend $S_\gamma$ to $B\left(\Gamma^0\right)$ setting the map $S:B\left(\Gamma^0\right)\to B\left(\Gamma^0\right)$ such that
    \begin{mylist}
        \item if $x\in\Gamma^0\setminus\Vbf$ let $\gamma\in\Ecal^+_x$ and $s\in\Scal_\gamma$ be such that $x=\gamma(s)$, then
        \begin{equation*}
            S[f](x)\coloneqq S_\gamma[f\circ\gamma](s);
        \end{equation*}
        \item if $x\in\Vbf$ we set $S$ through a two steps procedure:
        \begin{align*}
            \wtS[f](x)\coloneqq\,&\min\left\{S_\gamma[f]\left(\gamma^{-1}(x)\right):\gamma\in\Ecal_x^+\right\},\\
            S[f](x)\coloneqq\,&\min\left\{\wtS[f](x),f(x)+c_x\Delta t\right\}.
        \end{align*}
    \end{mylist}

    Now we consider the following evolutive problem:
    \begin{equation}\label{eq:globteikdisc}\tag{Discr}
        \left\{
        \begin{aligned}
            v(x,0)=\,&v_0(x),&&\text{if }x\in\Gamma^0,\\
            v(x,t)=\,&S[v(\cdot,t-\Delta t)](x),&&\text{if }t\ne0,\,(x,t)\in\Gamma^T,
        \end{aligned}
        \right.
    \end{equation}
    where $v_0:\Gamma\to\Rds$ is Lipschitz continuous.

    It is shown in \cite[Theorem 9.1]{CarliniSiconolfi23} that the solution to~\eqref{eq:globteikdisc} converges to the corresponding solution to~\eqref{eq:globteik}. We will provide in the next section an error estimate of this approximation.

    \subsection{Error Estimate}

    In order to provide an error estimate for the semi-Lagrangian scheme, we apply \cite[Theorem 4.3]{CarliniFestaForcadel20}. To make the paper self-contained, we restate it here as \cref{SLerr}.

    Fixed a flux limiter $c_x$, we denote with $\ovS$ a discrete numerical operator from $B\left(\Gamma^0\right)$ into itself and consider the following evolutive problem:
    \begin{equation}\label{eq:discsol}
        \left\{
        \begin{aligned}
            w(x,0)=\,&w_0(x),&&\text{if }x\in\Gamma^0,\\
            w(x,t)=\,&\ovS[w(\cdot,t-\Delta t)](x),&&\text{if }t\ne0,\,(x,t)\in\Gamma^T,
        \end{aligned}
        \right.
    \end{equation}
    where $w_0\in B\left(\Gamma^0\right)$.

    \begin{defin}
        Let $\Delta$ be an admissible pair. We say that the scheme~\eqref{eq:discsol} satisfies a \emph{consistency estimate} $E(\Delta)$ if we have,
        \begin{mynum}
            \item for any $\varphi\in C^2(\Gamma)$, $\gamma\in\Ecal^+$ and $s\in\Scal_\gamma\setminus\{0,|\gamma|\}$,
            \begin{equation*}
                \left|\frac{\varphi(\gamma(s))-\ovS[\varphi](\gamma(s))}{\Delta t}-H_\gamma(s,D(\varphi\circ\gamma)(s))\right|\le\max_{\gamma\in\Ecal}\left\|D^2(\varphi\circ\gamma)\right\|_\infty E(\Delta),
            \end{equation*}
            where $\|\cdot\|_\infty$ denotes the maximum norm;
            \item for any $\varphi\in C^2(\Gamma)$ and $x\in\Vbf$,
            \begin{equation*}
                \left|\frac{\varphi(x)-\ovS[\varphi](x)}{\Delta t}-F(x,\partial\varphi(x))\right|\le\max_{\gamma\in\Ecal}\left\|D^2(\varphi\circ\gamma)\right\|_\infty E(\Delta).
            \end{equation*}
            We recall that $F$ is defined as in~\eqref{eq:Fvertdef}, using $H_\gamma$ instead of $\wtH_\gamma$.
        \end{mynum}
    \end{defin}

    Let us focus on \cite[Theorem 4.3]{CarliniFestaForcadel20}. It requires the following conditions on each Lagrangian $L_\gamma$:
    \begin{Lassum}
        \item $L_\gamma$ is strictly convex, with respect to the second argument, and uniformly Lipschitz continuous;
        \item $\lim\limits_{|\lambda|\to\infty}\inf\limits_{s\in[0,|\gamma|]}\dfrac{L_\gamma(s,\lambda)}{|\lambda|}=\infty$;
        \item for all $\alpha>0$ there is a $C_\alpha>0$ such that
        \begin{equation*}
            \sup_{s\in[0,1]}\left|\inf_{|\lambda|\le\alpha}L_\gamma(s,\lambda)\right|\le C_\alpha.
        \end{equation*}
    \end{Lassum}
    We point out the proof of \cite[Theorem 4.3]{CarliniFestaForcadel20} exploits some consequences of these conditions (the Lipschitz continuity of the solutions to the evolutive problem and of the Hamiltonians $H_\gamma$) but are never directly used. Indeed, it is enough to assume~\ref{condconv}, \labelcref{condcoerc} and
    \begin{Hassum}[resume]
        \item each $H_\gamma$ is uniformly Lipschitz continuous.
    \end{Hassum}
    We recall that $\Hcal$ satisfies these conditions by \cref{modHprop}, while \cref{exunsolt} yields the Lipschitz continuity of the solutions to~\eqref{eq:globteikmod} whenever their initial datum is Lipschitz continuous. Moreover, even if~\cite{CarliniFestaForcadel20} employs the notion of solution given in \cref{IMdef}, \cref{defequiv} ensures that such notion is equivalent to \cref{deftsubsol}, which is employed by the scheme~\eqref{eq:globteikdisc}. These facts are summarized by the next \namecref{SLerr}.

    \begin{theo}\label{SLerr}
        Fixed a flux limiter $c_x$ and an admissible pair $\Delta$, let $u$ and $w$ be the solutions to~\eqref{eq:globteikmod} and~\eqref{eq:discsol} with Lipschitz continuous initial data $u_0$ and $w_0$, respectively, and flux limiter $c_x$. We further assume that the scheme
        \begin{mynum}
            \item is \emph{monotone}, i.e., given $g_1,g_2\in B\left(\Gamma^0\right)$ such that $g_1\le g_2$, we have $\ovS[g_1]\le\ovS[g_2]$;
            \item is \emph{invariant by addiction of constants}, i.e., $\ovS[g+A]=\ovS[g]+A$ for any constant $A$ and $g\in B\left(\Gamma^0\right)$;
            \item\label{en:SLerr3} is \emph{stable}, i.e., there exists a constant $K$, independent of $\Delta$, such that
            \begin{equation*}
                |w(x,t)-w_0(x)|\le Kt,\qquad\text{for any }(x,t)\in\Gamma^T;
            \end{equation*}
            \item satisfies a \emph{consistency estimate} $E(\Delta)$.
        \end{mynum}
        Then there exists a constant $C>0$, independent of $\Delta$, such that
        \begin{equation}\label{eq:SLerr.1}
            \max_{(x,t)\in\Gamma^T}|w(x,t)-u(x,t)|\le CT\left(\frac{E(\Delta)}{\sqrt{\Delta t}}+\sqrt{\Delta t}\right)+\max_{x\in\Gamma^0}|w_0(x)-u_0(x)|.
        \end{equation}
    \end{theo}

    \begin{rem}\label{CdepT}
        More precisely, see \cite[Theorem 4.3]{CarliniFestaForcadel20}, we have that the constant $C$ in \cref{SLerr} depends upon the Lipschitz constant of $u$, the stability constant $K$ in~\ref{en:SLerr3}, the Hamiltonian $\Hcal$, the flux limiter $c_x$ and the quantity $\|w_0-u_0\|_\infty$. It is not specified if $C$ is also independent of $T$, however, fixed a $\ovT>0$, $C$ can be chosen so that~\eqref{eq:SLerr.1} holds for all $T\le\ovT$.
    \end{rem}

    Let us focus on the semi-Lagrangian scheme. The first step for proving an error estimate using \cref{SLerr} is to show that the scheme is monotone, invariant and stable.

    \begin{lem}\label{SLstable}
        The semi-Lagrangian scheme~\eqref{eq:globteikdisc} is
        \begin{mynum}
            \item\label{en:SLstable1} \emph{monotone}, i.e., given $g_1,g_2\in B\left(\Gamma^0\right)$ such that $g_1\le g_2$, we have $S[g_1]\le S[g_2]$;
            \item\label{en:SLstable2} \emph{invariant by addiction of constants}, i.e., $S[g+A]=S[g]+A$ for any constant $A$ and $g\in B\left(\Gamma^0\right)$;
            \item \emph{stable}, i.e., let $v$ be the solution to~\eqref{eq:globteikdisc} with $\ell$--Lipschitz continuous datum $v_0$, then there exists a constant $K>0$, which depends only on $\ell$, the Hamiltonian $\Hcal$ and the flux limiter $c_x$, such that
            \begin{equation}\label{eq:SLstable.1}
                |v(x,t)-v_0(x)|\le Kt,\qquad\text{for any }(x,t)\in\Gamma^T.
            \end{equation}
        \end{mynum}
    \end{lem}
    \begin{proof}
        \Cref{en:SLstable1,en:SLstable2} are apparent. Moreover, given $g_1,g_2\in B\left(\Gamma^0\right)$ and setting $A_0\coloneqq\max\limits_{x\in\Gamma^0}|g_1(x)-g_2(x)|$,
        \begin{equation*}
            S[g_1]\le S[g_2+A_0]=S[g_2]+A_0,
        \end{equation*}
        thereby, for any $g_1,g_2\in B\left(\Gamma^0\right)$,
        \begin{equation}\label{eq:SLstable1}
            \max_{x\in\Gamma^0}|S[g_1](x)-S[g_2](x)|\le\max_{x\in\Gamma^0}|g_1(x)-g_2(x)|.
        \end{equation}
        Thanks to~\eqref{eq:SLstable1} we get, for any $0<i<N_T$,
        \begin{align*}
            \max_{x\in\Gamma^0}|v(x,(i+1)\Delta t)-v(x,i\Delta t)|=\;&\max_{x\in\Gamma^0}|S[v(\cdot,i\Delta t)](x)-S[v(\cdot,(i-1)\Delta t)](x)|\\
            \le\;&\max_{x\in\Gamma^0}|v(x,i\Delta t)-v(x,(i-1)\Delta t)|\le\max_{x\in\Gamma^0}|v(x,\Delta t)-v_0(x)|.
        \end{align*}
        It then follows, for any $0<m\le N_T$,
        \begin{equation*}
            \max_{x\in\Gamma^0}|v(x,m\Delta t)-v_0(x)|\le\sum_{i=1}^m\max_{x\in\Gamma^0}|v(x,i\Delta t)-v(x,(i-1)\Delta t)|\le m\max_{x\in\Gamma^0}|v(x,\Delta t)-v_0(x)|.
        \end{equation*}
        Hence, to conclude, it is enough to show that
        \begin{equation}\label{eq:SLstable2}
            |v(x,\Delta t)-v_0(x)|\le K\Delta t,\qquad\text{for any }x\in\Gamma^0,
        \end{equation}
        where $K$ depends on $\ell$, $\Hcal$ and $c_x$. We start assuming that $x\in\Vbf$ and $v(x,\Delta t)=v_0(x)+c_x\Delta t$, which trivially yields
        \begin{equation}\label{eq:SLstable3}
            |v(x,\Delta t)-v_0(x)|\le|c_x|\Delta t.
        \end{equation}
        Otherwise, we have by the definitions of $S$ and $\beta_0$ that there is a $\gamma\in\Ecal_x^+$, $s\in\Scal_\gamma$ and $\lambda\in[-\beta_0,\beta_0]$ such that $\gamma(s)=x$ and
        \begin{equation*}
            v(x,\Delta t)=I_\gamma[v_0\circ\gamma](s-\Delta t\lambda)+\Delta tL_\gamma(s,\lambda).
        \end{equation*}
        According to \cref{lipinterpol}, $I_\gamma[v_0\circ\gamma]$ is $\ell$--Lipschitz continuous as $v_0$, thus
        \begin{equation*}
            |v(x,\Delta t)-v_0(x)|\le\ell\lambda\Delta t+\Delta t|L_\gamma(s,\lambda)|\le\left(\ell\beta_0+\max_{s\in[0,|\gamma|],|\lambda|\le\beta_0}|L_\gamma(s,\lambda)|\right)\Delta t.
        \end{equation*}
        This and~\eqref{eq:SLstable3} show that~\eqref{eq:SLstable2} holds true with
        \begin{equation*}
            K\coloneqq\max\left\{\max_{x\in\Vbf}|c_x|,\max_{\gamma\in\Ecal^+}\left\{\ell\beta_0+\max_{s\in[0,|\gamma|],|\lambda|\le\beta_0}|L_\gamma(s,\lambda)|\right\}\right\}.
        \end{equation*}
    \end{proof}

    Next we prove that the scheme satisfies a consistency estimate.

    \begin{prop}\label{consisrates}
        Let $\Delta$ be an admissible pair with
        \begin{equation}\label{eq:consisrates.1}
            \Delta t\le\min_{\gamma\in\Ecal^+}\frac{\Delta_\gamma x}{\beta_0}.
        \end{equation}
        For any $\varphi\in C^2(\Gamma)$ the following estimates hold for~\eqref{eq:globteikdisc}:
        \begin{mynum}
            \item\label{en:consisrates1} if $\gamma\in\Ecal^+$ and $s\in\Scal_\gamma\setminus\{0,|\gamma|\}$, then
            \begin{equation*}
                \left|\frac{\varphi(\gamma(s))-S[\varphi](\gamma(s))}{\Delta t}-H_\gamma(s,D(\varphi\circ\gamma)(s))\right|\le C_0\max_{\gamma\in\Ecal}\left\|D^2(\varphi\circ\gamma)\right\|_\infty\min\left\{\Delta x,\frac{\Delta x^2}{\Delta t}\right\};
            \end{equation*}
            \item\label{en:consisrates2} if $x\in\Vbf$, then
            \begin{equation*}
                \left|\frac{\varphi(x)-S[\varphi](x)}{\Delta t}-F(x,\partial\varphi(x))\right|\le C_0\max_{\gamma\in\Ecal}\left\|D^2(\varphi\circ\gamma)\right\|_\infty\min\left\{\Delta x,\frac{\Delta x^2}{\Delta t}\right\},
            \end{equation*}
        \end{mynum}
        where $C_0$ is a positive constant independent of $\Delta$ and $\varphi$. In particular, the scheme~\eqref{eq:globteikdisc} satisfies a consistency estimate
        \begin{equation}\label{eq:consisrates.2}
            E(\Delta)\coloneqq C_0\min\left\{\Delta x,\frac{\Delta x^2}{\Delta t}\right\}.
        \end{equation}
    \end{prop}
    \begin{proof}
        First, we get from the error of the piecewise linear interpolation given in~\cite{CharlesDespresMehrenberger13,FalconeFerretti13}, for any $\gamma\in\Ecal^+$ and $s\in[0,|\gamma|]$,
        \begin{equation}\label{eq:consisrates1}
            \max_{\frac{s-|\gamma|}{\Delta t}\le\lambda\le\frac s{\Delta t}}|I_\gamma[\varphi\circ\gamma](s-\Delta t\lambda)-\varphi(\gamma(s-\Delta t\lambda))|\le A\left\|D^2(\varphi\circ\gamma)\right\|_\infty\min\left\{\Delta x\Delta t,\Delta x^2\right\},
        \end{equation}
        where $A$ is a positive constant independent of $\Delta$ and $\varphi$. Moreover, for every $\gamma\in\Ecal^+$ and $s\in\Scal_\gamma$, we obtain through~\eqref{eq:consisrates1} and a first order Taylor expansion
        \begin{multline}\label{eq:consisrates2}
            \varphi(\gamma(s))-S[\varphi](\gamma(s))\\
            \begin{aligned}
                =\,&\varphi(\gamma(s))-S_\gamma[\varphi\circ\gamma](s)\\
                =\,&\varphi(\gamma(s))-\min_{\frac{s-|\gamma|}{\Delta t}\le\lambda\le\frac s{\Delta t}}\{I_\gamma[\varphi\circ\gamma](s-\Delta t\lambda)+\Delta tL_\gamma(s,\lambda)\}\\
                \le\,&\max_{\frac{s-|\gamma|}{\Delta t}\le\lambda\le\frac s{\Delta t}}\{\Delta t\lambda D(\varphi\circ\gamma)(s)-\Delta tL_\gamma(s,\lambda)\}+C_0\left\|D^2(\varphi\circ\gamma)\right\|_\infty\min\left\{\Delta x\Delta t,\Delta x^2\right\}\\
                =\,&\max_{\frac{s-|\gamma|}{\Delta t}\le\lambda\le\frac s{\Delta t}}\{\Delta t\lambda D(\varphi\circ\gamma)(s)-\Delta tL_\gamma(s,\lambda)\}+\Delta t\left\|D^2(\varphi\circ\gamma)\right\|_\infty E(\Delta),
            \end{aligned}
        \end{multline}
        where $C_0$ is a positive constant independent of $\Delta$ and $\varphi$, and $E$ is defined in~\eqref{eq:consisrates.2}. If $s\in\Scal_\gamma\setminus\{0,|\gamma|\}$ we have by~\eqref{eq:consisrates.1} that $[-\beta_0,\beta_0]\subseteq\left[\dfrac{s-|\gamma|}{\Delta t},\dfrac s{\Delta t}\right]$, therefore
        \begin{align*}
            \varphi(\gamma(s))-S[\varphi](\gamma(s))\le\,&\Delta t\max_{|\lambda|\le\beta_0}\{\lambda D(\varphi\circ\gamma)(s)-L_\gamma(s,\lambda)\}+\Delta t\left\|D^2(\varphi\circ\gamma)\right\|_\infty E(\Delta)\\
            =\,&\Delta tH_\gamma(s,D(\varphi\circ\gamma))+\Delta t\left\|D^2(\varphi\circ\gamma)\right\|_\infty E(\Delta),
        \end{align*}
        where the last identity is due to the fact that $L(s,\lambda)=\infty$ when $\lambda>\beta_0$. In the same way
        \begin{equation*}
            \varphi(\gamma(s))-S[\varphi](\gamma(s))\ge\Delta tH_\gamma(s,D(\varphi\circ\gamma))-\Delta t\left\|D^2(\varphi\circ\gamma)\right\|_\infty E(\Delta).
        \end{equation*}
        These two inequalities prove~\ref{en:consisrates1}.\\
        Now let $x\in\Vbf$ and assume, without loss of generality, that $\gamma^{-1}(x)=0$ for all $\gamma\in\Ecal^+_x$. We have by definition
        \begin{equation*}
            \varphi(x)-S[\varphi](x)=\max\left\{-c_x\Delta t,\max_{\gamma\in\Ecal^+_x}\{\varphi(\gamma(0))-S_\gamma[\varphi\circ\gamma](0)\}\right\}.
        \end{equation*}
        Notice that~\eqref{eq:consisrates.1} yields
        \begin{equation}\label{weakcondition}
            \beta_0\le\dfrac{|\gamma|}{\Delta t},\qquad \text{for every }\gamma\in\Ecal^+,
        \end{equation}
        which in turn implies the inclusion $\left[-\dfrac{|\gamma|}{\Delta t},0\right]\subseteq[ -\beta_0,0]$. We therefore obtain from \eqref{eq:consisrates2}, for any $\gamma\in\Ecal_x^+$,
        \begin{align*}
            \varphi(\gamma(0))-S_\gamma[\varphi\circ\gamma](0)\le\,&\max_{-\frac{|\gamma|}{\Delta t}\le\lambda\le0}\{\Delta t\lambda D(\varphi\circ\gamma)(0)-\Delta tL_\gamma(0,\lambda)\}+\Delta t\left\|D^2(\varphi\circ\gamma)\right\|_\infty E(\Delta)\\
            =\,&\Delta t\max_{-\beta_0\le\lambda\le0}\{\lambda D(\varphi\circ\gamma)(0)-L_\gamma(0,\lambda)\}+\Delta t\left\|D^2(\varphi\circ\gamma)\right\|_\infty E(\Delta)\\
            =\,&\Delta tH_\gamma^-(0,D(\varphi\circ\gamma))+\Delta t\left\|D^2(\varphi\circ\gamma)\right\|_\infty E(\Delta),
        \end{align*}
        and
        \begin{equation}\label{eq:consisrates3}
            \begin{aligned}
                \varphi(x)-S[\varphi](x)\le\;&\Delta t\max\left\{-c_x,\max_{\gamma\in\Ecal^+_x}H^-_\gamma(0,p_\gamma\dot\gamma(0))\right\}+\Delta t\max_{\gamma\in\Ecal}\left\|D^2(\varphi\circ\gamma)\right\|_\infty E(\Delta)\\
                =\;&\Delta tF(x,\partial\varphi)+\Delta t\max_{\gamma\in\Ecal}\left\|D^2(\varphi\circ\gamma)\right\|_\infty E(\Delta).
            \end{aligned}
        \end{equation}
        Similarly we get
        \begin{equation*}
            \varphi(x)-S[\varphi](x)\ge\Delta tF(x,\partial\varphi)-\Delta t\max_{\gamma\in\Ecal}\left\|D^2(\varphi\circ\gamma)\right\|_\infty E(\Delta),
        \end{equation*}
        which together with~\eqref{eq:consisrates3} proves~\ref{en:consisrates2}.
    \end{proof}

    \begin{rem}
        Let us notice that the proof of \cref{consisrates}\ref{en:consisrates2} requires the inclusion
        \begin{equation*}
            \left[-\frac{|\gamma|}{\Delta t},0\right]\subseteq[ -\beta_0,0],\qquad\text{for every }\gamma\in\Ecal^+,
        \end{equation*}
        or, equivalently, \eqref{weakcondition}. This is verified by any $\Delta t $ small enough independently of $\Delta x$. In particular, the stronger condition~\eqref{eq:consisrates.1} is not necessary for this case.
    \end{rem}

    The \namecref{errCSaux} below is a consequence of \cref{SLerr,SLstable,consisrates}.

    \begin{theo}\label{errCSaux}
        Fixed a flux limiter $c_x$, $T>0$ and an admissible pair $\Delta$ satisfying~\eqref{eq:consisrates.1}, let $u$ and $v$ be a solution to~\eqref{eq:globteikmod} and~\eqref{eq:globteikdisc}, respectively, with Lipschitz continuous initial datum $u_0$ and flux limiter $c_x$. Then there exists a constant $C>0$, independent of $\Delta$, such that
        \begin{equation}\label{eq:errCSaux.1}
            |v(x,t)-u(x,t)|\le CT\left(\frac1{\sqrt{\Delta t}}\min\left\{\Delta x,\frac{\Delta x^2}{\Delta t}\right\}+\sqrt{\Delta t}\right),\qquad\text{for any }(x,t)\in\Gamma^T.
        \end{equation}
    \end{theo}

    \begin{rem}\label{Cindg}
        \Cref{CdepT} holds true for $C$ in \cref{errCSaux}. Following \cref{SLstable,exunsolt} we further have that $K$ in~\eqref{eq:SLstable.1}, and consequently $C$, can be taken independent of the initial datum $u_0$, using instead the Lipschitz constant of the solution $u$. In particular, \eqref{eq:errCSaux.1} holds true with the same constant $C$ for any collection of equi--Lipschitz continuous solutions to~\eqref{eq:globteikmod} with flux limiter $c_x$ and solutions to~\eqref{eq:globteikdisc} approximating them.
    \end{rem}

    Sometimes is useful to know the behavior of solutions to evolutive problems as $t\to\infty$, see for instance~\cite{Pozza23,AchdouCamilliCapuzzo-Dolcetta08}. Accordingly, it is convenient to have an error estimate as in~\eqref{eq:errCSaux.1} where $C$ is independent of $T$. However, as observed in \cref{CdepT}, it is unclear if the constant $C$ depends on $T$. We thus make an effort to prove a similar estimate for which the dependency on $T$ is explicit.

    \begin{cor}\label{errCSauxind}
        Fixed a flux limiter $c_x$, $T>0$ and an admissible pair $\Delta$ satisfying~\eqref{eq:consisrates.1}, let $u$ and $v$ be a solution to~\eqref{eq:globteikmod} and~\eqref{eq:globteikdisc}, respectively, with Lipschitz continuous initial datum $u_0$ and flux limiter $c_x$. Then there exists a constant $C>0$, independent of $\Delta$ and $T$, such that
        \begin{equation*}
            |v(x,t)-u(x,t)|\le CT\left(\frac1{\sqrt{\Delta t}}\min\left\{\Delta x,\frac{\Delta x^2}{\Delta t}\right\}+\sqrt{\Delta t}\right),\qquad\text{for any }(x,t)\in\Gamma^T.
        \end{equation*}
    \end{cor}
    \begin{proof}
        We proceed by induction and preliminarily fix $T_0\le1$. We then have by \cref{errCSaux} that
        \begin{equation*}
            |v(x,t)-u(x,t)|\le CT_0\left(\frac1{\sqrt{\Delta t}}\min\left\{\Delta x,\frac{\Delta x^2}{\Delta t}\right\}+\sqrt{\Delta t}\right),\qquad\text{for any }(x,t)\in\Gamma^{T_0},
        \end{equation*}
        where $C$ is a constant as in \cref{CdepT,Cindg}. In particular, $C$ can be assumed independent of $T$.\\
        Next we assume that $T=(n+1)T_0$ and that our claim holds true for $nT_0$, namely
        \begin{equation}\label{eq:errCSauxind1}
            |v(x,t)-u(x,t)|\le CnT_0\left(\frac1{\sqrt{\Delta t}}\min\left\{\Delta x,\frac{\Delta x^2}{\Delta t}\right\}+\sqrt{\Delta t}\right),\qquad\text{for any }(x,t)\in\Gamma^{nT_0}.
        \end{equation}
        For any fixed $t\in\Tcal^T$ with $t>nT_0$,
        \begin{equation}\label{eq:errCSauxind2}
            \max_{x\in\Gamma^0}|v(x,t)-u(x,t)|\le\max_{x\in\Gamma^0}\left|v(x,t)-v'(x,t-nT_0)\right|+\max_{x\in\Gamma^0}\left|v'(x,t-nT_0)-u(x,t)\right|,
        \end{equation}
        where $v'$ is the solution to~\eqref{eq:globteikdisc} on $\Gamma^{T_0}$ with initial datum $u(x,nT_0)$. Exploiting~\eqref{eq:SLstable1} we get by~\eqref{eq:errCSauxind1}
        \begin{equation}\label{eq:errCSauxind3}
            \begin{aligned}
                \max_{x\in\Gamma^0}\left|v(x,t)-v'(x,t-nT_0)\right|\le\;&\max_{x\in\Gamma^0}|v(x,nT_0)-u(x,nT_0)|\\
                \le\;&CnT_0\left(\frac1{\sqrt{\Delta t}}\min\left\{\Delta x,\frac{\Delta x^2}{\Delta t}\right\}+\sqrt{\Delta t}\right),
            \end{aligned}
        \end{equation}
        while \cref{Cindg} yields
        \begin{equation*}
            \max_{x\in\Gamma^0}\left|v'(x,t-nT_0)-u(x,t)\right|\le CT_0\left(\frac1{\sqrt{\Delta t}}\min\left\{\Delta x,\frac{\Delta x^2}{\Delta t}\right\}+\sqrt{\Delta t}\right),
        \end{equation*}
        where $C$ is the same constant independent of $T$ and $\Delta$ in~\eqref{eq:errCSauxind3}. This and~\eqref{eq:errCSauxind2} conclude our proof since $t$ is arbitrary.
    \end{proof}

    We stress out that, if $u_0$ is Lipschitz continuous and the modified Hamiltonian $\Hcal$ is chosen accordingly, the solutions to~\eqref{eq:globteikmod} and~\eqref{eq:globteik} with initial datum $u_0$ and flux limiter $c_x$ coincide by \cref{solmod}, thus \cref{errCSauxind} yields the next result.

    \begin{cor}\label{errCS}
        Fixed a flux limiter $c_x$, a Lipschitz continuous initial datum $u_0$, $T>0$ and an admissible pair $\Delta$ satisfying~\eqref{eq:consisrates.1}, let $u$ and $v$ be a solution to~\eqref{eq:globteik} and~\eqref{eq:globteikdisc}, respectively, with initial datum $u_0$ and flux limiter $c_x$. There exists a constant $C>0$, independent of $\Delta$ and $T$, such that
        \begin{equation*}
            |v(x,t)-u(x,t)|\le CT\left(\frac1{\sqrt{\Delta t}}\min\left\{\Delta x,\frac{\Delta x^2}{\Delta t}\right\}+\sqrt{\Delta t}\right),\qquad\text{for any }(x,t)\in\Gamma^T.
        \end{equation*}
    \end{cor}

    \begin{rem}
        Let $\{\Delta_n\}_{n\in\Nds}$ be an infinitesimal sequence of admissible pairs, and denote by $v_n$ the solution to~\eqref{eq:globteikdisc} with $\Delta_n$ instead of $\Delta$. \Cref{errCS} shows that, if these pairs satisfy~\eqref{eq:consisrates.1} and
        \begin{equation*}
            \lim_{n\to\infty}\min\left\{\frac{\Delta_n x}{\sqrt{\Delta_n t}},\frac{\Delta_n x^2}{\Delta_n t^{\frac32}}\right\}=0,
        \end{equation*}
        then $v_n$ converges to the solution $u$ to~\eqref{eq:globteik}. Furthermore, by choosing $\Delta t=O(\Delta x)$, we obtain a convergence estimate of order one-half.\\
        Without assuming~\eqref{eq:consisrates.1} convergence still holds (as proved in \cite[Theorem 9.1]{CarliniSiconolfi23}), although no error estimate is provided, if $\Delta_n x\le\Delta_n t$ for each $n$ big enough and
        \begin{equation}\label{def:ICFL}
            \lim_{n\to\infty}\frac{\Delta_n x}{\Delta_n t}=0.
        \end{equation}
        This kind of result is not surprising: a similar conclusion is found in~\cite{CarliniFestaForcadel20}, where a convergence estimate of order one-half is proven under a hyperbolic CFL condition, while, without this assumption, a convergence result is still achieved but without an associated error estimate. The numerical results, shown in the next \lcnamecref{sec:numerics}, confirm that condition~\eqref{eq:consisrates.1} is sufficient but not necessary for convergence.
    \end{rem}

    \section{Impact of Time Step Restriction Violations}\label{sec:numerics}

    We examine the time step restriction~\eqref{eq:consisrates.1}, required for our error estimate, in two numerical simulations, adapted from \cite[Section 9.1]{CarliniSiconolfi23}. In the first test the Hamiltonians are independent of $s$, while in the second one the Hamiltonians depend on $s$.

    We start by observing that the constant $\beta_0$ can be chosen such that $\beta_0 > \beta_{\gamma}$ for any $\gamma$, where $\beta_{\gamma}$ is the Lipschitz constant of $\widetilde H_{\gamma}$ over a given compact set. For a more detailed discussion on the definition of $\beta_0$, refer to \cite[Appendix A]{CarliniSiconolfi23}. In the simple case where the Hamiltonians do not depend on the state variable, \eqref{eq:consisrates.1} coincides with the CFL condition given in \cite[Section 1.3]{GuerandKoumaiha19} and \cite[Section 1.2]{CostesequeLebacqueMonneau14}. It is important to note that the CFL condition in~\cite{GuerandKoumaiha19,CostesequeLebacqueMonneau14} is necessary to guarantee monotonicity. In our case, however, the condition plays a different role. Specifically, it is needed to prove the consistency error estimate in \cref{consisrates}\ref{en:consisrates1}. As typical in semi-Lagrangian schemes, the scheme is monotone by construction without any condition on the parameters. If condition~\eqref{eq:consisrates.1} is violated, the scheme~\eqref{eq:globteikdisc} remains monotone but may lose accuracy, as demonstrated in the first of the two following examples. We measure the scheme performance by computing the following errors at the final time $T$:
    \begin{equation}\label{err}
        E^\infty \coloneqq\max_{x \in \Gamma^0} |u(x,T) - v(x,T)|, \qquad E^1 \coloneqq\sum_{x \in \Gamma^0} |u(x,T) - v(x,T)| \Delta x,
    \end{equation}
    where $v$ is the approximated solution defined by the scheme~\eqref{eq:globteikdisc} and $u$ indicates the exact solution of problem~\eqref{eq:globteikmod}, when it is known, or a reference solution, when the exact solution is not available.

    Our numerical tests are performed on the following network: let $\Gamma \subset\Rds^2$ be chosen as a triangle with vertices
    \begin{equation*}
        z_1 \coloneqq(0,0), \qquad z_2 \coloneqq(0,1), \qquad z_3 \coloneqq\left(\frac{1}{2}, \frac{1}{2}\right),
    \end{equation*}
    with arcs
    \begin{equation*}
        \gamma_1(s) \coloneqq s \, z_2, \qquad \gamma_2(s) \coloneqq s \, z_3, \qquad \gamma_3(s) \coloneqq(1-s) \, z_2 + s \, z_3,
    \end{equation*}
    and the corresponding reversed arcs, as shown in \cref{fig:Network}.

    \begin{figure}[hbt!]
        \centering
        \begin{tikzpicture}[scale=3]
            \draw[->] (-0.5,0) -- (1.5,0) node[right] {$x_1$};
            \draw[->] (0,-0.5) -- (0,1.5) node[above] {$x_2$};

            \node[below left] at (0,0) {$z_1$};
            \node[above left] at (0,1) {$z_2$};
            \node[right ] at (0.5,0.5) {$z_3$};

            \draw[thick, postaction={decorate, decoration={markings, mark=at position 0.5 with {\arrow{>}}}}] (0,0) -- (0,1)
            node[midway, left=3pt] {$\gamma_1$};

            \draw[thick, postaction={decorate, decoration={markings, mark=at position 0.5 with {\arrow{>}}}}] (0,0) -- (0.5,0.5)
            node[midway, right=3pt] {$\gamma_2$};

            \draw[thick, postaction={decorate, decoration={markings, mark=at position 0.5 with {\arrow{>}}}}] (0,1) -- (0.5,0.5)
            node[midway, above=3pt] {$\gamma_3$};

            \filldraw[black] (0,0) circle (0.5pt);
            \filldraw[black] (0,1) circle (0.5pt);
            \filldraw[black] (0.5,0.5) circle (0.5pt);
        \end{tikzpicture}
        \caption{Network $\Gamma \subset \Rds^2 $}\label{fig:Network}
    \end{figure}
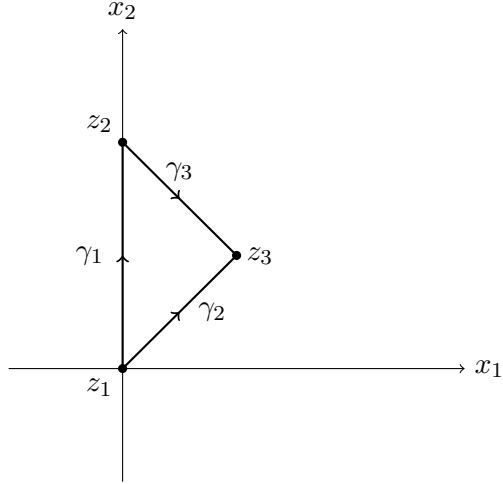

    \subsection{Test with Hamiltonians Independent of \texorpdfstring{$s$}{s}}\label{ssec:test1}

    The cost functions are chosen as $ L_{\gamma_i}(\lambda) \coloneqq\frac{\lambda^2}{2} $ and the admissible flux limiters are set as $c_{z_1}=c_{z_2}=c_{z_3}\coloneqq-5$. The initial condition is $ v_0 \coloneqq0 $. This example has an exact solution for $ t \ge \frac{1}{\sqrt{10}} $, given by:
    \begin{equation*}
        u(x,t) =\left\{
        \begin{aligned}
            &\left(\frac12 - \left|x_2 - \frac12\right|\right)\sqrt{10} - 5t, && \text{if } x \in \gamma_1,\\
            &\left(\frac{1}{2\sqrt{2}} - \left||x| - \frac{1}{2\sqrt{2}}\right|\right)\sqrt{10} - 5t, && \text{if } x \in \gamma_2,\\
            &\left(\left|\left|x - \left(\frac{1}{2}, \frac{1}{2}\right)\right| - \frac{1}{2\sqrt{2}}\right| + \frac{1}{2\sqrt{2}}\right)\sqrt{10} - 5t, && \text{if } x \in \gamma_3,
        \end{aligned}
        \right.
    \end{equation*}
    where $ x = (x_1, x_2) $ is the variable in the physical space $ \Rds^2 $. We remark that, since the Hamiltonians do not depend on the space variable, the characteristics are straight lines and are computed exactly by our scheme. In addition, the exact solution $u$ is affine and then the interpolation error in~\eqref{eq:S_gamma} is almost zero.

    In \cref{tab:Hindipsolex}, we show the errors~\eqref{err} computed with several refinements of the space grid and three different values of $ \Delta t $, given by $ \Delta t = (\Delta x)^{4/5}/2 $, $ \Delta t = (\Delta x)/2 $ and $\Delta t = (\min_{\gamma}\Delta_\gamma x)/4 $. The last choice satisfies~\eqref{eq:consisrates.1}, since for this problem we can choose $\beta_0=4$. Even though the first two choices of $\Delta t$ violate condition~\eqref{eq:consisrates.1}, numerical convergence is observed (columns 2--5). We point out that the first choice satisfies~\eqref{def:ICFL}.

    In \cref{fig:Hindipsolex}, we display three plots representing the absolute value of the difference between the exact solution $ u $ at time $T=1$ and the approximated solution $ v $ computed at time $ T = 1 $ with $ \Delta x = 0.05 $ for $ \Delta t = (\Delta x)^{4/5}/2 $, $ \Delta t = \Delta x/2 $, and $ \Delta t = (\min_{\gamma}\Delta_\gamma x)/4 $. We observe no errors on the vertices in all the graphics. Indeed, in this case the exact solution $u$ on the vertices is given by $u(z_i,t)=c_{z_i} t$ and the scheme~\eqref{eq:globteikdisc} computes these values exactly. In the left plot, the largest errors occur near the vertices, where the minimization is performed on an asymmetric interval since the discrete characteristics are truncated by construction, and these errors propagate along the arcs. In the middle plot, we can observe the same phenomena, however the smaller time step reduce the errors. In the right plot, where condition~\eqref{eq:consisrates.1} is satisfied, the largest errors occur in the middle of the arcs and are produced only by the approximation of the minimum in~\eqref{eq:S_gamma}. This is because condition~\eqref{eq:consisrates.1} ensures that the minimization is performed, for each node of the mesh, on an interval large enough to contain the optimal control, as shown in the proof of \cref{consisrates}\ref{en:consisrates1}. This explains why the errors in columns 6--7 have much smaller order of magnitude with respect to the other columns.

    \begin{table}[hbt!]
        \renewcommand\arraystretch{1.1}

        \centering
        \begin{tabular}{lllllll}
            \toprule
            ${\Delta x}$ &$E^{\infty}$ &$E^{1}$ &$E^{\infty}$ & $E^{1}$ & $E^{\infty}$ &$E^{1}$ \\
            \midrule
            $1.00\cdot 10^{-1}$ &$1.62\cdot 10^{-1}$ &$3.95 \cdot 10^{-2}$ &$5.08\cdot 10^{-2}$ &$1.59\cdot 10^{-2}$ &$1.20\cdot 10^{-5}$ &$1.07\cdot 10^{-7}$ \\
            $5.00\cdot 10^{-2}$ &$1.04\cdot 10^{-1}$ &$1.94 \cdot 10^{-2}$ &$2.07\cdot 10^{-2}$ &$7.10\cdot 10^{-3}$ &$4.92\cdot 10^{-7}$ &$1.02\cdot 10^{-7}$\\
            $2.50\cdot 10^{-2}$ &$6.33\cdot 10^{-2}$ &$7.60 \cdot 10^{-3}$ &$9.27\cdot 10^{-3}$ &$3.39\cdot 10^{-3}$ &$3.66\cdot 10^{-7}$ &$6.74\cdot 10^{-8}$ \\
            $1.25\cdot 10^{-2}$ &$4.08\cdot 10^{-2}$ &$3.54 \cdot 10^{-3}$ &$4.35\cdot 10^{-3}$ &$1.65\cdot 10^{-3}$ &$8.87\cdot 10^{-8}$ &$1.72\cdot 10^{-8}$\\
            $6.25\cdot 10^{-3}$ &$2.58\cdot 10^{-2}$ &$1.51 \cdot 10^{-3}$ &$2.17\cdot 10^{-3}$ &$8.23\cdot 10^{-4}$ &$7.08\cdot 10^{-9}$ &$2.03\cdot 10^{-9}$\\
            \bottomrule
        \end{tabular}
        \caption{Errors for example in \cref{ssec:test1}, computed at time $T=1$. Columns 2--3 refer to $\Delta t=(\Delta x)^{4/5}/2$. Columns 4--5 refer to $\Delta t=\Delta x /2$. Columns 6--7 refer to $\Delta t=(\min_\gamma \Delta_\gamma x)/4$.}\label{tab:Hindipsolex}
    \end{table}

    \begin{figure}[hbt!]
        \centering

        \includegraphics[width=0.32\textwidth]{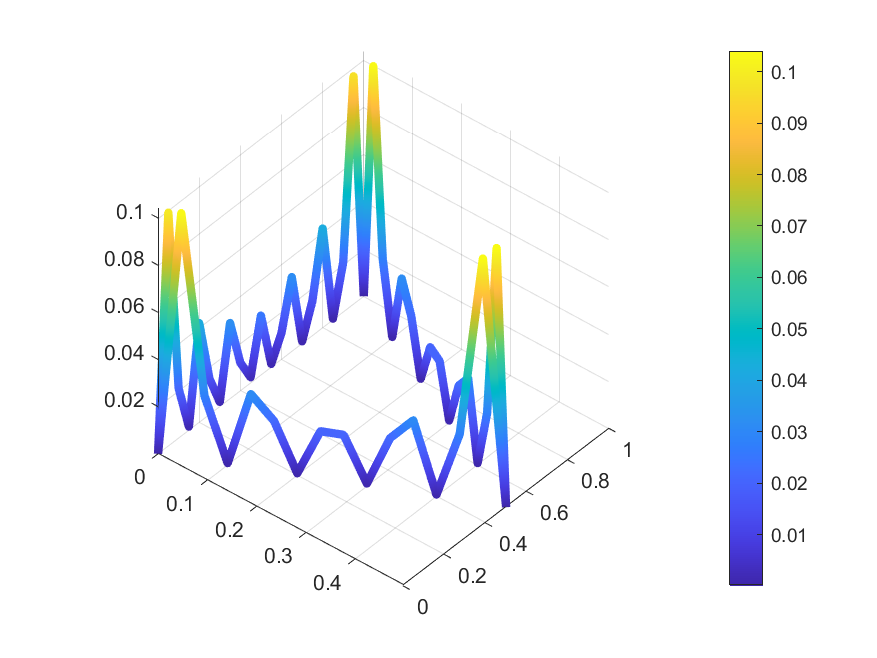}
        \includegraphics[width=0.32\textwidth]{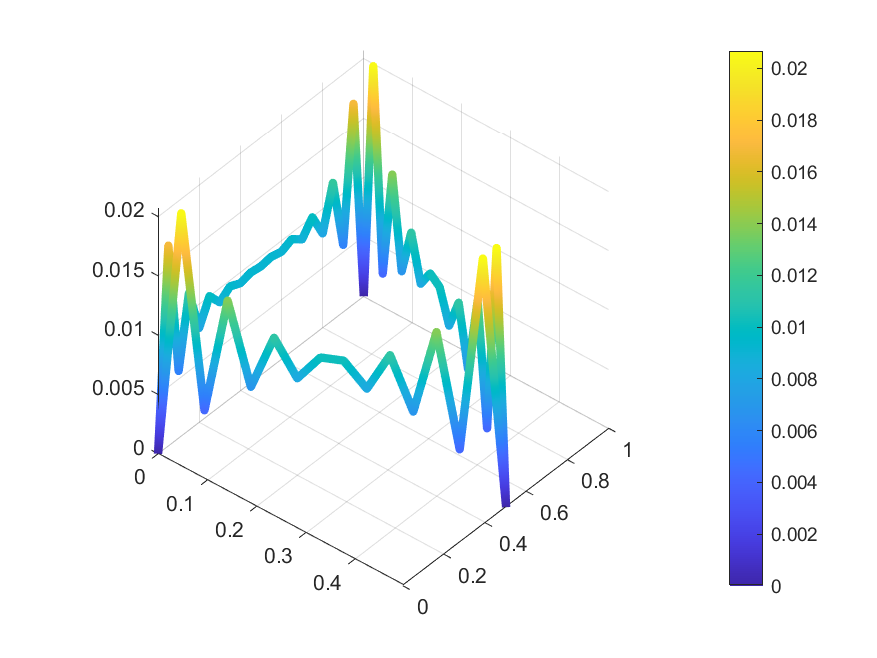}
        \includegraphics[width=0.32\textwidth]{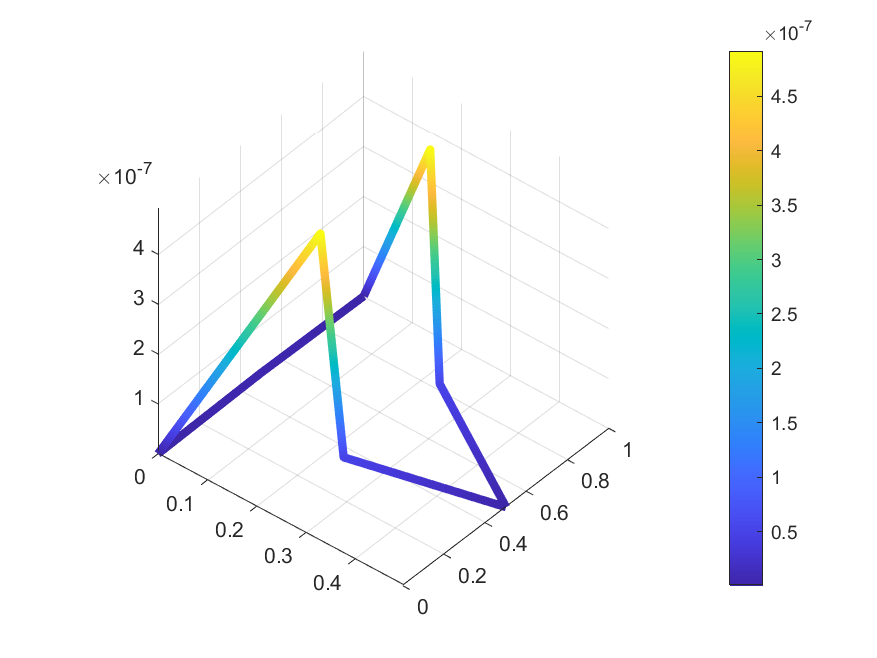}
        \caption{Errors $|u(x,T)-v(x,T)|$ for \cref{ssec:test1}, computed at $T=1$ with $ \Delta x = 0.05$, for $ \Delta t = (\Delta x)^{4/5}/2 $ (left), $ \Delta t = \Delta x/2 $ (center), $ \Delta t = (\min_{\gamma}\Delta_\gamma x)/4 $ (right).}\label{fig:Hindipsolex}
    \end{figure}

    \subsection{Test with Hamiltonians Dependent on \texorpdfstring{$s$}{s}}\label{ssec:test2}

    In this example, the cost functions are taken depending on $s$, as
    \begin{align*}
        L_{\gamma_1}(s,\lambda) \coloneqq\;& \frac{|\lambda|^2}{2}+ 5|\gamma_1(s)-(0.5,0.5)|^2, \\
        L_{\gamma_2}(s,\lambda) \coloneqq\;& \frac{|\lambda|^2}{2}+ 5|\gamma_2(s) -(0.5,0.5)|^2+ 10 (\gamma_2(s)_1)^2, \\
        L_{\gamma_3}(s,\lambda) \coloneqq\;&\frac{|\lambda|^2}{2}+ 5|\gamma_3(s)-(0.5,0.5)|^2+10 (\gamma_3(s)_1)^2,
    \end{align*}
    where $\gamma(s)_1$ indicates the first component of the point $\gamma(s)$. The cost in $\gamma_1$ is optimized at $z_3$, whereas the costs in $\gamma_2$ and $\gamma_3$ increase with the distance of the physical point from $z_3$ and from the line $x_1=0$. We choose as admissible flux limiters $c_{z_1}=c_{z_2}=c_{z_3}\coloneqq1$ and initial condition $v_0(x)\coloneqq1-x$. For this problem the exact solution is not explicit, therefore we compute a reference solution $u$ using the scheme with $\Delta x=10^{-3}$ and $\Delta t=(\min_{\gamma}\Delta_\gamma x)/3$. In \cref{tab:Hdip}, we show the errors~\eqref{err} computed with several refinements of the space grid and three different values of $ \Delta t $, given by $ \Delta t = (\Delta x)^{4/5}/2 $, $ \Delta t = (\Delta x)/2 $ and $\Delta t = (\min_{\gamma}\Delta_\gamma x)/3 $. As in the previous test, the last choice satisfies~\eqref{eq:consisrates.1}, since for this problem we can choose $ \beta_0 =3 $. The first two choices of $\Delta t$ violate condition~\eqref{eq:consisrates.1}, however the first one satisfies requirement~\eqref{def:ICFL}. Unlike in the previous test, for any choice of $\Delta x$, we observe almost the same errors corresponding to the three different choices of $\Delta t$. This is because, in addition to the errors related to the choice of $\Delta t$ discussed in the previous test, here two truncation errors are added. These are the truncation error of the Euler method, used in scheme~\eqref{eq:S_gamma} to approximate the characteristics which are not affine since the Hamiltonians depend on $s$, and the interpolation error, since the exact solution in this case is not affine.

    \begin{table}[hbt!]
        \renewcommand\arraystretch{1.1}

        \centering
        \begin{tabular}{lllllll}
            \toprule
            ${\Delta x}$ &$E^{\infty}$ &$E^{1}$ &$E^{\infty}$ & $E^{1}$ & $E^{\infty}$ &$E^{1}$ \\
            \midrule
            $1.00\cdot 10^{-1}$ &$4.41\cdot 10^{-2}$ &$5.11 \cdot 10^{-2}$ &$4.50\cdot 10^{-2}$ &$5.22\cdot 10^{-2}$ &$4.55\cdot 10^{-2}$ &$5.24\cdot 10^{-2}$ \\
            $5.00\cdot 10^{-2}$ &$2.43\cdot 10^{-2}$ &$2.77 \cdot 10^{-2}$ &$2.28\cdot 10^{-2}$ &$2.78\cdot 10^{-2}$ &$2.29\cdot 10^{-2}$ &$2.78\cdot 10^{-2}$\\
            $2.50\cdot 10^{-2}$ &$1.37\cdot 10^{-2}$ &$1.63 \cdot 10^{-2}$ &$1.45\cdot 10^{-2}$ &$1.47\cdot 10^{-2}$ &$1.49\cdot 10^{-2}$ &$1.47\cdot 10^{-2}$ \\
            $1.25\cdot 10^{-2}$ &$8.22\cdot 10^{-3}$ &$9.73 \cdot 10^{-3}$ &$7.65\cdot 10^{-3}$ &$7.53\cdot 10^{-3}$ &$8.07\cdot 10^{-3}$ &$7.61\cdot 10^{-3}$\\
            $6.25\cdot 10^{-3}$ &$4.46\cdot 10^{-3}$ &$4.97 \cdot 10^{-3}$ &$3.54\cdot 10^{-3}$ &$3.52\cdot 10^{-3}$ &$3.76\cdot 10^{-3}$ &$3.57\cdot 10^{-3}$\\
            \bottomrule
        \end{tabular}
        \caption{Errors for example in \cref{ssec:test2}, computed at time $T=1$. Columns 2--3 refer to $\Delta t=(\Delta x)^{4/5}/2$. Columns 4--5 refer to $\Delta t=\Delta x /2$. Columns 6--7 refer to $\Delta t=(\min_\gamma \Delta_\gamma x)/3$.}\label{tab:Hdip}
    \end{table}

    These results indicate that the time step restriction~\eqref{eq:consisrates.1} is sufficient but not necessary for convergence and suggest that a-priori convergence rates could be achievable under weaker conditions.

    \printbibliography[heading=bibintoc]

\end{document}